\documentclass[11pt]{article}

\usepackage{graphicx}
\usepackage{amsthm,amsmath,amssymb,verbatim,fancyhdr}
\usepackage{array,url}
\usepackage{fullpage}
\usepackage{color}

\newtheorem{thm}{Theorem}
\newtheorem{lem}[thm]{Lemma}
\newtheorem{cor}[thm]{Corollary}
\newtheorem{conj}[thm]{Conjecture}
\newtheorem{clm}{Claim}
\newtheorem{problem}{Problem}

\def\VEC#1#2#3{#1_{#2},\ldots,#1_{#3}}
\newcommand\KG{\mathop{\rm KG}}

\newcommand{\di}{\displaystyle}
\def\qed{\ifhmode\unskip\nobreak\hfill$\Box$\bigskip\fi \ifmmode\eqno{Box}\fi}

\newcommand\R{\mathcal R}

\begin{document}

\title{Triangle-free subgraphs with large fractional chromatic number}

\author{
  Bojan Mohar\thanks{Supported in part by the NSERC Discovery Grant R611450 (Canada),
  by the Canada Research Chairs program,
  and by the Research Project J1-8130 of ARRS (Slovenia).}\\[1mm]
  Department of Mathematics\\
  Simon Fraser University\\
  Burnaby, BC, Canada\\
  {\tt mohar@sfu.ca}
\and
  Hehui Wu\thanks{Part of this work was done while the author was a PIMS Postdoctoral Fellow at the Department of Mathematics, Simon Fraser University, Burnaby, B.C.}\\[1mm]
  Shanghai Center for Math.\ Sci.\\
  Fudan University\\
  Shanghai, China\\
  {\tt hhwu@fudan.edu.cn}
}

\maketitle

\begin{abstract}
  It is well known that for any integers $k$ and $g$, there is a graph with chromatic number at least $k$ and girth at least $g$. In 1960's, Erd\H{o}s and Hajnal conjectured that for any $k$ and $g$, there exists a number $h(k,g)$, such that every graph with chromatic number at least $h(k,g)$ contains a subgraph with chromatic number at least $k$ and girth at least $g$. In 1977, R\"{o}dl proved the case for $g=4$ and arbitrary $k$. We prove the fractional chromatic number version of R\"{o}dl's result.
\end{abstract}

\section{Outline}

A well known result, proved by Erd\H{o}s in 1950s \cite{Erd}, tells us that for every $k$ and $g$, there exists a graph with chromatic number at least $k$ and girth at least $g$. In 1960s, Erd\H{o}s and Hajnal proposed the following \cite{EH}.

\begin{conj}[Erd\H{o}s and Hajnal (1969)]
\label{conj:ErdosHajnal}
For every positive integers $k$ and $g$, there exists an integer $h(k,g)$ such that every graph $G$ with $\chi(G)\ge h(k,g)$ contains a subgraph with chromatic number at least $k$ and girth at least $g$.
\end{conj}

In 1977, R\"{o}dl~\cite{Rodl} proved the conjecture for $g=4$ and arbitrary $k$. The special case when $g=4$ speaks about triangle-free subgraphs. This is the only nontrivial case for which the Erd\H{o}s-Hajnal conjecture has been confirmed.

\begin{thm}[R\"odl (1977)]
\label{thm:Rodl}
For every positive integer $k$, there exists an integer $f(k)$ such that if $\chi(G)\ge f(k)$ then $G$ contains a triangle-free subgraph $H$ with $\chi(H)=k$.
\end{thm}

Let $\mathcal{I}(G)$ be the family of all independent sets of $G$, and let $\mathcal{I}(G,v)$ be the family of all those independent sets which contain the vertex $v$. For each independent set $I$, consider a nonnegative real variable $y_I$. The \emph{fractional chromatic number} of $G$, denoted by $\chi_f(G)$, is the minimum value of $$\sum_{I\in\mathcal{I}(G)} y_I, \quad \hbox{ subject to } \sum_{I\in\mathcal{I}(G,v)} y_I \ge 1 \hbox{ for each } v\in V(G).
$$

Erd\H{o}s actually proved in \cite{Erd} that there exist graphs with large girth and large fractional chromatic number.
In this paper, we prove the fractional chromatic number version of R\"{o}dl's result.

\begin{thm}
\label{thm:main}
For every real number $x\ge1$, there exists a positive number $k(x)$ such that every graph $G$ with $\chi_f(G)\ge k(x)$ contains a triangle-free subgraph $H$ with $\chi_f(H)\ge x$.
\end{thm}

Given this result, we put forward the following.

\begin{conj}
\label{conj:1}
For every real number $x\ge1$ and any integer $g\ge3$, there exists a positive number $k(x,g)$ such that every graph $G$ with $\chi_f(G)\ge k(x,g)$ contains a subgraph $H$ of girth at least $g$ and with $\chi_f(H)\ge x$.
\end{conj}

There is a natural question whether the chromatic number in the Erd\H{o}s-Hajnal conjecture should be replaced by the fractional chromatic number. This leads us to graphs with large chromatic number and small fractional chromatic number.
The most vulnerable examples in this respect might be Kneser graphs, for which the conjecture should be tested first.
As our second contribution, we show in Section \ref{sect:blow-up} that the Erd\H{o}s-Hajnal Conjecture \ref{conj:ErdosHajnal} holds for Kneser graphs. The difference from the required value of the chromatic number of Kneser graphs versus the chromatic number of their subgraphs of large girth and chromatic number is very reasonable. See Theorem \ref{THM:KGBU} and Corollary \ref{Cor:KGBU}.

\subsection*{Related problems}

The Erd\H{o}s-Hajnal conjecture and the fact that random graphs provide examples of graphs of large girth and large fractional chromatic number motivates the following question. Let $G$ be any graph with large (fractional) chromatic number. Is it true that a random subgraph of $G$ will also have large (fractional) chromatic number (with high probability). Boris Bukh \cite{Bukh} asked a weaker question, whether a random subgraph of a graph with chromatic number $n$ has chromatic number comparably large to the chromatic number of a random graph of order $n$. More precisely, if $G$ is a graph and $p \in [0,1]$, we let $G_p$ denote a subgraph of $G$ where each edge of $G$ appears in $G_p$ independently with probability $p$.

\begin{problem}[Bukh]
\label{prb:Bukh}
Does there exist a constant $c$ so that for every graph $G$
$${\mathbb E}(\chi(G_{1/2})) > c \frac{\chi(G)}{\log \chi(G)}\,?$$
\end{problem}

The methods in this paper led us to a solution of the fractional version of Problem \ref{prb:Bukh}, see \cite{MW_Bukh}.

To put the Erd\H{o}s-Hajnal conjecture in a bigger context, we cannot abide a related conjecture that was proposed in in 1983 by Carsten Thomassen.

\begin{conj}[Thomassen \cite{Th83}]
\label{conj:Thomassen}
For every positive integers $k$ and $g\ge3$, there exists a positive number $d(k,g)$ such that every graph with average degree at least $d(k,g)$ contains a subgraph of girth at least $g$ and average degree at least $k$.
\end{conj}

Although this conjecture is still wide open, there are partial results.
K\"uhn and Osthus \cite{KO} proved the conjecture for $g\le6$. A different proof of the same result is given in Dellamonica et al. \cite{DKMR11}.
Pyber et al. \cite{PRS94} and Dellamonica and R\"odl \cite{DR11} proved the conjecture under additional assumption when the maximum degree is not too large in terms of the average degree.

Here we propose a conjecture that is a weakening of Conjectures \ref{conj:ErdosHajnal} and \ref{conj:Thomassen}.

\begin{conj}
\label{conj:ErdHajThomassen}
For every positive integers $k$ and $g\ge3$, there exists a positive number $c(k,g)$ such that every graph with chromatic number at least $c(k,g)$ contains a subgraph of girth at least $g$ and with average degree at least $k$.
\end{conj}

While Thomassen's conjecture is trivial when we replace the girth condition by requesting that the \emph{odd girth} is at least $g$ (meaning that the graph has no odd cycles of length less than $g$), the Erd\H{o}s-Hajnal conjecture for odd girth may be of interest.

\begin{conj}
\label{conj:oddgirth}
For every positive integers $k$ and $g$, there exists an integer $h(k,g)$ such that every graph $G$ with $\chi(G)\ge h(k,g)$ contains a subgraph with chromatic number at least $k$ and odd girth at least $g$.
\end{conj}

\section{Proof of Theorem \ref{thm:main}}

The proof of Theorem \ref{thm:main} uses the tools presented in this section.

Let $\VEC v1n$ be a linear ordering of the vertices of a graph $G$, i.e. we consider $v_i<v_j$ whenever $i<j$. For a vertex $v$, let
$$
   N^L(v)=\{u\mid uv\in E(G), u<v\},
$$
that is, the set of the neighbors of $v$ that appear before $v$. R\"{o}dl's proof of Theorem \ref{thm:Rodl} is based on the following lemma (see \cite{Rodl}).

\begin{lem}\label{R:LNC}
If\/ $\chi(G)> k^t$, and $\chi(N^L(v))\le t$ for every $v\in V(G)$, then there exists a triangle-free subgraph $H$ of\/ $G$ with $\chi(H)> k$.
\end{lem}

The original proof of Lemma \ref{R:LNC} is elegant and very short, but it cannot be applied to the fractional chromatic number. Our major effort is to extend the above claim to the fractional chromatic number setup.

For a function $w: V(G)\to \mathbb{R}$, and a vertex-set $A$, we write $w(A) = \sum_{v\in A} w(v)$.
The \emph{fractional independence number}, denoted by $\alpha_f(G)$, is the minimum value of
\begin{equation}
   \max\{w(I)\mid I\in \mathcal{I}(G)\}
   \label{eq:fin}
\end{equation}
where the minimum of (\ref{eq:fin}) is taken over all non-negative weight functions $w$ with $w(V)=n$. By the linear programming duality we have the fact that $\chi_f(G)=\frac{n}{\alpha_f(G)}$, where $n=|V(G)|$. So we can consider our problem as a fractional independence number problem.

The following lemma is our statement on the fractional independence number analogous to Lemma~\ref{R:LNC}. It uses the following function
$$
   \di f(x,l) = x \left(\frac{\Gamma(xl^7+1)}{\Gamma(x+1)}\right)^3
$$
defined for real numbers $x\ge 1$ and $l\ge1$, where $\Gamma(\cdot)$ stands for the Euler gamma function. We also fix a \emph{weight function} $w:V(G)\to \mathbb{R}^+$ on the vertices of $G$, with the assumption that $w(V)\ne0$. The weight function $w$ defines a linear order $<$ on $V(G)$ such that $u<v$ implies that $w(u)\ge w(v)$; that is, the vertices are ordered starting with those with largest weight.
In the proof of Theorem \ref{thm:main}, the function $w$ will be the one minimizing (\ref{eq:fin}) in the definition of the fractional independence number of $G$.

\begin{lem}
\label{LNFC}
Let $x\ge59$ and $l\ge1$.
Suppose that\/ $w(I)\le \frac{w(V)}{f(x,l)}$ for every $I\in \mathcal{I}(G)$, and that $\chi_f(N^L(v))\allowbreak \le l$ for every $v\in V(G)$. Then $G$ contains a triangle-free spanning subgraph $H$, such that $w(I)\le \frac{w(V)}{x}$ for every $I\in \mathcal{I}(H)$.
In particular, $\chi_f(H) \ge x$.
\end{lem}

\begin{proof}
If $1\le l < 2$, then $G$ is triangle-free. Namely, if $G$ would contain a triangle $uvz$ with $u<v<z$, then $uv$ would be an edge in $N^L(z)$, contradicting the assumption that $\chi_f(N^L(z))\le l < 2$. Moreover, the condition that $w(I)\le w(V)/f(x,l)$ for every $I\in {\mathcal I}(G)$ implies that $\chi_f(G)\ge f(x,l)\ge f(x,1)=x$. Thus we can take $H=G$ if $1\le l < 2$.

For $l\ge2$, Lemma \ref{LNFC} is proved by induction together with Lemma \ref{LEM:RS}. The latter one uses Lemma \ref{LNFC} in its proof with $l' = l(1-\frac{1}{6(x+1)})$ playing the role of $l$ and $x'=x+1$ playing the role of $x$. Then Lemma \ref{LNFC} is proved for $l$ and $x$ by applying the other lemma. In order to see that this inductive proof scheme works, we need to argue that by applying the transformation from $(l,x)$ to $(l',x+1)$ consecutively a finite number of times, we eventually obtain the value $l'$ which is between 1 and 2.
To see this, let $l_0=l$ and for $i\ge1$, let $l_i = l_{i-1}(1-\frac{1}{6(x+i)})$ be the value of $l'$ after $i$ steps, $i=1,2,\dots$.  The process stops if $l_i$ becomes smaller than 2.
Observe that whenever this happens, $l_{i-1}\ge2$ and thus $l_i \ge 2(1-\frac{1}{6(x+1)}) > 1$.
Moreover, note that $l_i \le l_{i-1} - \frac{2}{6(x+i)}$ as long as $l_{i-1}\ge2$. Thus,
$$
   l_i \le l_0 - \frac{1}{3}\Bigl(\frac{1}{x+1} + \frac{1}{x+2} + \cdots \frac{1}{x+i}\Bigr).
$$
Since the harmonic series diverges, the value of $l_i$ eventually becomes smaller than 2.
This justifies the inductive scheme of our proof. The rest of the proof of Lemma \ref{LNFC} is the inductive step proved for any $l\ge2$; that part of the proof is given after the proof of Lemma \ref{LEM:RS}.

Let $\VEC v1n$ be the enumeration of the vertices in the non-increasing order according to the weight function $w$. Given a vertex-set $A$, let $A_k$ be the set of the first $k$ elements in $A$ according to this ordering. We extend this notion to all positive real numbers by setting
$A_s := A_{\lfloor s\rfloor}$.

For $s\in {\mathbb R}^+$, a nonempty subset $X$ of a vertex-set $Y$ is said to be \emph{$s$-principal\/} in $Y$ if $X\subseteq Y_{s|X|}$. That is, if $X$ has size $m$, then all elements of $X$ are within the first $\lfloor sm\rfloor$ vertices in $Y$. On the other hand, a subset $X$ of $Y$ is \emph{$s$-sparse} in $Y$ if $X$ contains no $s$-principal subset in $Y$. When the hosting set $Y$ for $s$-principal or $s$-sparse is not specified, by default it is the whole vertex-set $V$.

The next claim from \cite{MW} about the total weight of an $s$-sparse set will be essential for us. We include the proof for completeness.

\begin{clm}\label{CLM:sparseislight}
Let $s>0$ be a real number.
If $X$ is an $s$-sparse subset of\/ $Y$, then $w(X)\le \frac{1}{s}\,w(Y)$.
\end{clm}

\begin{proof}
Let $\VEC y1r$ be the non-decreasing order of the elements of $Y$ with $r=|Y|$, and let $\VEC x1m$ be the ordering of $X$ with $m=|X|$. Since $X$ is an $s$-sparse subset of $Y$, we have $x_i\not\in Y_{si}$. Hence for $1\le i\le m$, $w(x_i)\le w(y_j)$ if $1\le j\le si$. Moreover, since $x_i\in Y\setminus Y_{si}$, we also have $w(x_i)\le w(y_j)$ for $j=\lceil si\rceil$.

For a real parameter $z\in(0,|Y|]$, define $f(z)=y_{\lceil z\rceil}$. Then $f(z)\ge w(x_1)$ for $0<z\le s$, $f(z)\ge w(x_2)$ for $s<z\le 2s$, \dots, $f(z)\ge w(x_m)$ for $(m-1)s < z \le ms$.
Therefore,
$$
    s\,w(X) = s\sum_{i=1}^m w(x_i) \le \int_0^{sm} f(z)dz \le \sum_{j=1}^{\lceil sm\rceil} w(y_j) \le w(Y),
$$
which gives what we were aiming to prove.
\end{proof}

Let $A\subseteq V$ be a vertex-set and $v\in A$. Let $L_A(v)$ be the subgraph of $G$ induced by the neighbors of $v$ in $A$ that appear before $v$.
The set $A$ is \emph{$(x,l)$-reducible} if it satisfies the following conditions:
\begin{enumerate}
\item[(i)] $w(A)\ge \frac{w(V)}{x(x+1)^2}$ \ and
\item[(ii)] $\chi_f(L_A(v))\le l(1-\frac{1}{6(x+1)})$ \ for every $v\in A$.
\end{enumerate}

\begin{lem}\label{LEM:RS}
Suppose that $x\ge3$ and $l\ge2$ are real numbers.
Suppose that Lemma \ref{LNFC} holds for $x'=x+1$ and $l' = l-\frac{l}{6(x+1)}$.
Suppose that every independent set $I$ in $G$ satisfies $w(I)\le w(V)/f(x,l)$.
For any $(x,l)$-reducible set $A$, there is a triangle-free subgraph $H_A$ with $V(H_A)=A$ such that any independent set of $H_A$ has weight at most $\frac{w(A)}{x+1}$.
\end{lem}

\begin{proof}
Note that $f(x,1)=x$. Moreover, $f(x, l)$ satisfies the following recursive bound:
\begin{equation}
  f(x,l)\ge x(x+1)^2 f(x',l').
  \label{eq:f(x,l)}
\end{equation}
A short calculation combined with the fact that the gamma function $\Gamma(t)$ is increasing for $t\ge 2$ shows that (\ref{eq:f(x,l)}) holds if and only if $(x+1)(1-\tfrac{1}{6(x+1)})^7 \le x$.
By using elementary calculus, it is easy to verify that this inequality holds for every $x\ge3$. This confirms (\ref{eq:f(x,l)}).

For any independent set $I\subseteq A$, we have
$$
  w(I)\le \frac{w(V)}{f(x,l)}\le \frac{w(A)}{\frac{f(x,l)}{x(x+1)^2}}\le \frac{w(A)}{f(x',l')}.
$$
By our assumption, Lemma \ref{LNFC} holds for $x'$ and $l'$. Therefore, there is a triangle-free subgraph $H_A$ with $V(H_A)=A$ such that any independent set in $H_A$ has weight at most $\frac{w(A)}{x+1}$.
\end{proof}

In what follows, we will prove Lemma \ref{LNFC} with the help of Lemma~\ref{LEM:RS}. By our inductive scheme described before, we may assume that Lemma \ref{LNFC} holds for $x'$ and $l'$, and thus the conclusion of Lemma \ref{LEM:RS} can be used.

Let $\R$ be a maximal collection of pairwise disjoint $(x,l)$-reducible sets and let $R=\cup\R$ be the union of all these sets. Then the complement $\overline{R} = V\setminus R$ contains no reducible subsets. Applying Lemma~\ref{LEM:RS} for each $A\in \R$, we can find a triangle-free spanning subgraph $H_A$ of $G(A)$, such that for any independent set $I\subset V(H_A) = A$, we have $w(I)\le \frac{w(A)}{x+1}$. We let $H_0 = \cup_{A\in\R} H_A$. Then every $I\in {\mathcal I}(H_0)$ is disjoint union of independent sets in subgraphs $H_A$ ($A\in\R$) and thus
\begin{equation}
  w(I)\le \sum_{A\in\R} \frac{w(A)}{x+1} = \frac{w(R)}{x+1}.
  \label{eq:all reducible}
\end{equation}

Let $L_G(v)$ denote the graph $L_{V(G)}(v)$. By the assumption of Lemma \ref{LNFC}, we have $t = t(v) := \chi_f(L_G(v))\le l$. Let $\mathcal{I}(v)$ be the collection of independent sets of $L_G(v)$. There exists a weight function $u:\mathcal{I}(v)\to [0,1]$, such that any vertex in $L_G(v)$ is covered by independent sets with total weight at least $1$, and total weight of $\mathcal{I}(v)$ is $t$. For a set $A$ containing $v$, we say $v$ is \emph{type $1$ in} $A$ if the total $u$-weight of those sets in $\mathcal{I}(v)$ that are disjoint from $A$ is at most $\frac{t}{6(x+1)}$; otherwise $v$ is \emph{type $2$ in} $A$. Let $T_1(A)$ be the collection of type 1 vertices in $A$ and $T_2(A)$ be the collection of type 2 vertices in $A$.

A nonempty vertex-set $A\subseteq \overline{R}$ is said to be \emph{dense} if $A$ is $(x+1)$-principal in $\overline{R}$ and $|T_2(A)|\le \frac{|A|}{x+1}$.

\begin{lem}
\label{lem:no dense subset}
$G$ contains a triangle-free spanning subgraph $H$ such that no dense subset of $\overline{R}$ is independent in $H$.
\end{lem}

\begin{proof}
We start by taking the subgraph $H_0$ with vertex-set $R$ defined above. Then we define a subgraph $H$
by adding the vertices of $\overline{R}$ and some edges from these vertices to the rest of the graph by using the following random choice.
For each vertex $v\in \overline{R}$, randomly pick an independent set $I$ from $\mathcal{I}(v)$ according to their weight $u$, and then add the edges between $I$ and $v$ to $H$.
Now, the lemma follows from Claims \ref{clm:3} and \ref{clm:4} that are proved below.
\end{proof}

\begin{clm}
\label{clm:3}
$H$ is a triangle-free subgraph of $G$.
\end{clm}

\begin{proof}
Suppose there is a triangle with vertices $v_i, v_j, v_k$. Since $H_0$ is triangle-free, not all of these vertices are in $R$. Let $v_r\in \{v_i, v_j, v_k\}\cap \overline{R}$ be the one with $r$ largest possible. Then the two edges from $v_r$ to the other two vertices in $\{v_i, v_j, v_k\}$ have been added by the random choice at $v_r$. However, $v_r$ was joined to an independent set, and thus the other two vertices cannot be adjacent. This is in contradiction to $v_iv_jv_k$ being a triangle in $H$.
\end{proof}

\begin{clm}
\label{clm:4}
The probability that every dense subset of $\overline{R}$ contains an edge in $H$ is positive.
\end{clm}

\begin{proof}
For a dense subset $A\subseteq \overline{R}$ of cardinality $k$, at least $k - \frac{k}{x+1} = \frac{kx}{x+1}$ vertices of $A$ are type 1. For each type 1 vertex $v$ of $A$, the total weight of sets in ${\mathcal I}(v)$ is $t=t(v)\le l$, and the total weight of those that are out of $A$ is at most $\frac{t}{6(x+1)}$. Therefore, the selected independent set $I$ from ${\mathcal I}(v)$ has no vertices in $A$ with probability at most $\frac{1}{6(x+1)}$. Therefore, the probability that $A$ is an independent set in $H$ is at most $(\frac{1}{6(x+1)})^{\frac{kx}{x+1}}$.

As every dense subset of $\overline{R}$ is $(x+1)$-principal, there are at most ${\lfloor k(x+1)\rfloor \choose k}$ dense sets of order $k$. By Stirling's Formula, $k!\ge \sqrt{2\pi k}\,(\frac ke)^k > (\frac ke)^k$. Thus, we have
$${\lfloor k(x+1)\rfloor \choose k}< \frac{(k(x+1))^k}{k!} < \frac{(k(x+1))^k}{(\frac ke)^k} = e^k(x+1)^k.$$
Therefore, the probability that some dense set of order $k$ is an independent set in $H$ is less than
\begin{equation}
   e^k(x+1)^k\Bigl(\frac{1}{6(x+1)}\Bigr)^{\frac{kx}{x+1}} =
   \Bigl(e\,6^{-\frac {x}{x+1}}(x+1)^{\frac{1}{x+1}}\Bigr)^k.
\label{eq:dense independent}
\end{equation}
If $x\ge 59$, then $e\cdot 6^{-\frac {x}{x+1}} (x+1)^{\frac{1}{x+1}} < 1/2$, and thus
the right-hand side in (\ref{eq:dense independent}) is less than $2^{-k}$.
So we have the probability that some dense set is an independent set in $H$ is less than $\sum_{k\ge1} 2^{-k}<1$. With positive probability, every dense subset of $\overline{R}$ contains an edge in $H$.
\end{proof}

\begin{lem}
\label{lem:9}
If a set $S\subseteq \overline{R}$ contains no dense subset, then
$$w(S)\le \frac{w(\overline{R})}{x+1}+\frac{w(V)}{x(x+1)}.$$
\end{lem}

\begin{proof}
Let $S = \{s_1,s_2,\dots,s_{|S|}\}$, where the enumeration is consistent with the weight $w$, i.e. $w(s_i)\ge w(s_j)$ whenever $i\le j$.
For $1\le k\le |S|$, we have $S_k = \{s_1,\dots,s_k\}$ is not dense. This means that either $S_k$ is not $(x+1)$-principal in $\overline{R}$, or $|T_2(S_k)|>\frac{|S_k|}{x+1}$ (and $S_k$ is $(x+1)$-principal).

Let $S'=\{s_k \mid S_k \mbox{ is not $(x+1)$-principal in } \overline{R} \}$.
By the definition of sparse sets, it is easy to see that $S'$ is $(x+1)$-sparse in $\overline{R}$. By Claim~\ref{CLM:sparseislight}, we have $w(S')\le \frac{1}{x+1}w(\overline{R})$.

Let $S''=S\setminus S'=\{s_k \mid S_k \mbox{ is $(x+1)$-principal in } \overline{R} \}$. We can enumerate the elements in $S''$ as $s_{i_1},s_{i_2},s_{i_3},\dots$, where $1\le i_1 < i_2 < i_3 < \cdots$.
If $S_k$ is $(x+1)$-principal in $\overline{R}$, since $S$ does not contain any dense set and by the definition of dense set, we have $|T_2(S)\cap S_k|=|T_2(S_k)|> \frac{|S_k|}{x+1}=\frac{k}{x+1}$. This implies that $T_2(S)$ contains elements $s_{r_1},s_{r_2},s_{r_3},\dots$, where $1\le r_1 < r_2 < r_3 < \cdots$ such that $r_1\le i_1$, $r_2\le i_{\lceil x+1\rceil}$, $r_3\le i_{\lceil 2(x+1)\rceil}$, etc. As a consequence we obtain the following:
\begin{eqnarray*}
  (x+1)w(s_{r_1}) &\ge& \sum_{r=1}^{x+1} w(s_{i_r}), \\
  (x+1)w(s_{r_2}) &\ge& \sum_{r=x+2}^{2(x+1)} w(s_{i_r}),
\end{eqnarray*}
etc. (The sums in these inequalities are to be understood in the same way as in the proof of Claim \ref{CLM:sparseislight}, where we used integration in order to be precise: If the value $r$ of the upper (or lower) bound in the summation is not an integer, then we add partial value of the corresponding weight
$w(s_{i_{\lceil r\rceil}})$ ($w(s_{i_{\lfloor r\rfloor}})$) proportional to the distance from the ``floor'' (the ``ceiling'') of the value.)
By summing up these inequalities, we get $(x+1) w(T_2(S)) \ge w(S'')$.

As $T_2(S)$ is not reducible, we have one of the following outcomes: either $w(T_2(S)) < \frac{w(V)}{x(x+1)^2}$ or $\chi_f(L_{T_2(S)}(v)) > l(1-\frac{1}{6(x+1)})$ for some $v\in T_2(S)$. We claim that the second outcome contradicts our assumptions. By the definition of
$T_2(S)$, for each $v\in T_2(S)$, the total weight of independent sets in $\mathcal{I}(v)$ that are disjoint from $S$ is at least $\frac{t}{6(x+1)}$ (where $t=t(v)\le l$), the other independent sets have total weight at most $t(1-\frac{1}{6(x+1)})$. This corresponds to a fractional coloring of $L_{T_2(S)}(v)$ with order $t(1-\frac{1}{6(x+1)}) \le l(1-\frac{1}{6(x+1)}) < \chi_f(L_{T_2(S)}(v))$, which is a contradiction. So we must have $w(T_2(S)) < \frac{w(V)}{x(x+1)^2}$.
Hence $w(S'')\le (x+1)w(T_2(S))<\frac{w(V)}{x(x+1)}$.

Now we have $w(S)=w(S') + w(S'') < \frac{w(\overline{R})}{x+1} + \frac{w(V)}{x(x+1)}.$
\end{proof}

Let $I$ be an independent set in $H$. By Lemma \ref{lem:no dense subset}, $I$ contains no dense subsets.
By (\ref{eq:all reducible}) and Lemma \ref{lem:9} we have
\begin{eqnarray*}
  w(I) &=& w(I\cap R)+w(I\cap \overline{R}) \\
       &\le& \frac{w(R)}{x+1}+\frac{w(\overline{R})}{x+1}+\frac{w(V)}{x(x+1)} \\
       &=& w(V)\Bigl(\frac 1{x+1}+\frac{1}{x(x+1)}\Bigr) = \frac{w(V)}{x}.
\end{eqnarray*}
This completes the proof of Lemma \ref{LNFC}.
\end{proof}

We are ready to give the proof of Theorem \ref{thm:main}.

\begin{proof}[Proof of Theorem \ref{thm:main}]
Clearly, we may assume that $x\ge59$ by setting $k(x)=k(59)$ for every $x<59$.

Let $r(x)$ be the smallest integer such that there exists a triangle-free graph $Q$ of order $r(x)$ with $\chi_f(Q)\ge x$.
(It follows by the known bounds on Ramsey numbers $R(3,t)$ \cite{Kim} and a result by Ajtai, Koml\'os, and Szemer\'edi \cite{AKS} that $r(x)=\Theta(x^2/\log x)$.)
Define $k_0(x)=x$ and set $k_t(x)=f(x,k_{t-1}(x))$ for $t=1,2,\dots,r(x)$. Finally, let $k(x)=k_{r(x)}(x)$.

Let $t=r(x)$ and let $G_t=G$ be a graph with $\chi_f(G)\ge k(x)=k_t(x)$. Consider the weight function $w_t$ which minimizes (\ref{eq:fin}) in the definition of the fractional independence number of $G_t$. If there is a vertex $v_t$ such that $\chi_f(N^L(v_t)) > k_{t-1}(x)$, then we consider the subgraph induced on $N^L(v_t)$ and decrease $t$ by 1. This new graph defines the new weight function and for the decreased value of $t$ we repeat the same test whether there is a vertex $v_t$ such that $\chi_f(N^L(v_t)) > k_{t-1}(x)$. If we decrease $t$ all the way down to 1, then the vertices $v_1,v_2,\dots,v_{r(x)}$ form a complete subgraph of $G$, and by the definition of $r(x)$, there is a triangle-free subgraph isomorphic to $Q$, whose fractional chromatic number is at least $x$. This gives the desired outcome of the theorem.

On the other hand, if the process stops at $t>1$, we have obtained a graph $G_t$ such that $\chi_f(G_t) \ge k_t(x)=f(x,k_{t-1}(x))$ and for every vertex $v$, we have $\chi_f(N^L(v_t)) \le k_{t-1}(x)$. By Lemma \ref{LNFC}, $G_t$ (and hence also $G$) contains a triangle-free subgraph $H$ with $\chi_f(H) \ge x$.
\end{proof}

\section{Blow-ups and Kneser graphs}
\label{sect:blow-up}

The fact that the Erd\H{o}s-Hajnal Conjecture is so resistant and the results of this paper open the question whether graphs with bounded fractional chromatic number (and large chromatic number) would still satisfy the conjecture. The most natural examples of such graphs are Kneser graphs. In this section we give the proof that the Erd\H{o}s-Hajnal Conjecture holds for them.

Let us recall that the vertex-set of the Kneser graph $\KG(n,k)$ consists of all $k$-sets of elements of $\{1,\dots,n\}$, and two such sets are adjacent if they are disjoint. It may be assumed that $k\le n/2$. It is known that $\chi(\KG(n,k)) = n-2k+2$ and that $\chi_f(\KG(n,k)) = n/k$.

Given a graph $H$, the \emph{blow-up} of $H$ with \emph{power} $m$, denoted by $H^{(m)}$, is the graph obtained from $H$ by replacing each vertex by an independent set of size $m$ (called the \emph{blow-up\/} of the vertex), and for each edge $xy$ in $H$, the two blow-ups of $x$ and $y$ form a complete bipartite graph $K_{m,m}$. The subgraph of $H^{(m)}$ replacing an edge $xy$ of $H$ is isomorphic to $K_{m,m}$ and will be referred to as the \emph{blow-up} of that edge.

We have the following statement.

\begin{thm}\label{THM:BUT}
Suppose $G$ is a graph with $\Delta(G)\le \Delta$ and $\chi(G)>x$.
Suppose that $m$ is an integer that is larger than $x(x\Delta)^{2g-4}$. Then there exists a subgraph $H$ of $G^{(m)}$ with girth more than $g$ and chromatic number more than~$x$.
\end{thm}

There are a few existing papers, for example \cite{BS, ZHU}, in which a result similar to Theorem \ref{THM:BUT} was proved. (The corresponding results in \cite{BS, ZHU} were applied to a construction of uniquely colorable graphs of large girth.) But the bound for the blow-up power $m$ in \cite{BS} and \cite{ZHU} is too large for our purpose as it depends on the number of vertices of $G$ instead of the maximum degree.

To prove Theorem~\ref{THM:BUT}, we will need the following fact from \cite{BS}.

\begin{lem}\label{lem:claim4}
Given a graph $G$ with $\chi(G)>x$, let $H$ be a subgraph of $G^{(m)}$. Suppose that for any edge $ab\in E(G)$ and for any subsets $X, Y$ contained in the blow-ups of $a$ and $b$, respectively, with $|X|\ge\frac{m}{x}$, $|Y|\ge\frac{m}{x}$, there is an edge between $X$ and $Y$ in $H$. Then $\chi(H) > x$.
\end{lem}

In the proof of Theorem~\ref{THM:BUT} we will take a random subgraph $H$ of $G^{(m)}$, obtained by selecting each edge independently with probability $(\frac{m}{x})^{\frac 1{4l}-1}$, and will prove that with positive probability $H$ has no short cycles, and for any edge $ab\in E(G)$ and for any pair $X, Y$ contained in the respective blow-ups of $a$ and $b$, and with $|X|\ge \frac{m}{x}$, $|Y|\ge\frac{m}{x}$, there is an edge between $X$ and $Y$ in $H$. To prove this, we will use the following asymmetric form of the Lov\'{a}sz Local Lemma.

\begin{thm}[Lov\'{a}sz Local Lemma]
\label{thm:LLL}
Let $\mathcal{A} = \{ A_1, \ldots, A_n \}$ be a finite set of events in the probability space $\Omega$. For $A \in \mathcal{A}$ let $\Gamma(A)$ denote a subset of $\mathcal{A}$ such that $A$ is independent from the collection of events $\mathcal{A} \setminus (\{A \} \cup \Gamma(A))$. If there exists an assignment of real numbers $y : \mathcal{A} \to (0,1)$ to the events such that
$$
  \forall A \in \mathcal{A} : \Pr(A) \leq y(A) \prod\nolimits_{B \in \Gamma(A)} (1-y(B))
$$
then the probability of avoiding all events in $\mathcal{A}$ is positive. In particular,
$$
  \Pr\left(\overline{A_1} \wedge \ldots \wedge \overline{A_n} \right) \geq \prod\nolimits_{A \in \mathcal{A}} (1-y(A)).
$$
\end{thm}

\begin{proof}[Proof of Theorem~\ref{THM:BUT}]
We can assume $x\not= 2$, since there is an odd cycle with length more than $g$ in $G^{(m)}$.

Let $s=\frac{m}{x}$, $\lambda=\frac 1{4g}$ and $p=s^{\lambda-1}$.
Let $H$ be a random subgraph of $G^{(m)}$ obtained by picking each edge in $G^{(m)}$ independently with probability $p$.

If $C$ is a cycle in $G^{(m)}$ of length at most $g$, let $A_C$ be the event that all edges of $C$ appear in $H$. Then $Pr(A_C)=p^{|C|}=s^{(\lambda-1)|C|}$.

Let $B$ be a copy of $K_{s,s}$ as a subgraph of a blow-up of an edge in $G$, let $A_B$ be the event that $H$ contains none of the edges of $B$. Then $Pr(A_B) = (1-p)^{s^2}\approx e^{-ps^2} = e^{-s^{1+\lambda}}$.

To prove that there exists a subgraph of $G^{(m)}$ with girth at least $g$ and chromatic number at least $x$, we will use Lemma~\ref{lem:claim4}. We just need to show
$$Pr\biggl(\biggl(\,\bigwedge_{|C|\le g}\overline{A_C}\biggr)\wedge\biggl(\,\bigwedge_{B \cong K_{s,s}}\overline{A_B}\biggr)\biggr)>0.$$
This will be confirmed by applying the asymmetric form of the Lov\'{a}sz Local Lemma (Theorem \ref{thm:LLL}) to the two types of events together.

Suppose that $C$ is a cycle of length $|C|\le g$. As the maximum degree in $G^{(m)}$ is at most $m\Delta =sx\Delta $, there are at most $|C|(sx\Delta)^{j-2}$ cycles of length $j$ in $G^{(m)}$ that share edges with $C$, and there are at most $|C|{m\choose s}^2$ copies of $K_{s,s}$ that share edges with $C$.

Suppose that $B$ is a copy of $K_{s,s}$. There are at most $s^2(sx\Delta)^{j-2}$ cycles of length $j$ in $G^{(m)}$ that share edges with $B$, and there are at most ${m\choose s}^2$ copies of $K_{s,s}$ that share edges with $B$.

For the Lov\'{a}sz Local Lemma, let $y_0=y(A_B)=e^{-0.5s^{1+\lambda}} \approx Pr(A_B)^{0.5}$, and for each cycle $C$ of length $|C|\le g$, let $y_{|C|}=y(A_C)=Pr(A_C)^{1-\lambda}=s^{-(1-\lambda)^2|C|}<s^{(2\lambda-1)|C|}$. We just need to show:

(1) $Pr(A_C)\le y_{|C|}\left(\prod_{3\le j\le g}(1-y_j)^{|C|(sx\Delta)^{j-2}}\right)(1-y_0)^{|C|{m\choose s}^2}$;

(2) $Pr(A_B)\le y_0\left(\prod_{3\le j\le g}(1-y_j)^{s^2(sx\Delta)^{j-2}}\right)(1-y_0)^{{m\choose s}^2}$.

Let us take the logarithm on each side of the above inequalities, and use the fact that $0.9\log (1-z) > -z$ (when $z$ is close to $0$ as it appears to be in our case when $z=y_0$ or $y_j$). After simplifying, we see that it suffices to verify the following inequalities:

(3) $\di 0.9\lambda(1-\lambda)\log s \ge \sum_{3\le j\le g} s^{(2\lambda-1)j}(sx\Delta)^{j-2} +
e^{-0.5s^{1+\lambda}}{sx\choose s}^2$.

(4) $\di 0.4s^{1+\lambda}\ge \sum_{3\le j\le g} s^{(2\lambda-1)j}s^2(sx\Delta)^{j-2}+e^{-0.5s^{1+\lambda}}{sx\choose s}^2$.

\smallskip
\noindent
In order to prove (4), we first observe that:
\begin{eqnarray*}
  \sum_{3\le j\le g} s^{(2\lambda-1)j}s^2(sx\Delta)^{j-2}
  &=& \sum_{3\le j\le g} s^{2\lambda j}(x\Delta)^{j-2} \\
  &<& 1.1s^{2\lambda g}(x\Delta)^{g-2} \\
  &=& 1.1s^{0.5}(x\Delta)^{g-2}.
\end{eqnarray*}
By the assumption of the theorem, $s > (x\Delta)^{2g-4}$.
Thus, $(x\Delta)^{g-2} < s^{0.5}$, and we have
$$\sum_{3\le j\le g} s^{(2\lambda-1)j}s^2(sx\Delta)^{j-2}<1.1s$$
Similarly, we also have
$\sum_{3\le j\le g} s^{(2\lambda-1)j}(sx\Delta)^{j-2}<1.1s^{-1}$, which will be used to prove (3).

As ${sx\choose s}^2 < (ex)^{2s}$, we have
$$e^{-0.5s^{1+\lambda}}{sx\choose s}^2 < (e^{-0.5s^{\lambda}}(ex)^2)^s=(e^{-0.5s^{\lambda}+2+2\log x})^s.$$
As $s>(x\Delta)^{2g-4}$, we have $s^{\lambda}>(x\Delta)^{\frac {2g-4}{4g}}>4(1+\log x)$, hence $e^{-0.5s^{1+\lambda}}{sx\choose s}^2<1.$

As $0.9\lambda(1-\lambda)\log s\ge 0.9(1-\lambda)\frac {2g-4}{4g}\log (x\Delta)\ge 1.1s^{-1}+1$ and $0.4s^{1+\lambda}\ge 1.1s+1$, we conclude that both inequalities (3) and (4) are true.

In summary, if $m=sx\ge x(x\Delta)^{2g-4}$, by the asymmetric form of the Lov\'{a}sz Local Lemma, the event that $H$ has no cycles of length at most $g$ and every $K_{s,s}$ as a subgraph of a blow-up of an edge has at least one edge, has positive probability. Hence, by Lemma~\ref{lem:claim4}, the corresponding subgraph $H$ of $G^{(m)}$ has chromatic number more than $x$.
\end{proof}

The following result from \cite[Theorem 3.3]{MW} shows that Kneser graphs contain blow-ups of smaller Kneser graphs with large power.

\begin{thm}\label{THM:KGBU}
Let $n,k,t$, and $x$ be nonnegative integers such that $0<k<n$ and $x<kt$.
The Kneser graph\/ $\KG(nt,kt-x)$ contains the blow-up of $\KG(n,k)$ with power ${k(t-1)\choose x}$ as a subgraph. Furthermore, when $x<t$, it contains the blow-up of $\KG(n,k)$ with power ${kt\choose x}$, and when $x=t$, it contains the blow-up of $\KG(n,k)$ with power ${kt\choose x}-k$.
\end{thm}

From Theorem~\ref{THM:BUT}, we know that graphs that are blow-ups of smaller graphs with sufficiently large power satisfy the Erd\H{o}s-Hajnal Conjecture. In particular, Kneser graphs are such examples. This can be used to derive the main result of this section.

\begin{cor}
\label{Cor:KGBU}
The Erd\H{o}s-Hajnal Conjecture holds for Kneser graphs.
\end{cor}

\begin{proof}
Let $k$ and $g$ be the parameters from the Erd\H{o}s-Hajnal Conjecture.
Let $\KG(2n, n-2x)$ be a Kneser graph with large chromatic number. Since $\chi(\KG(2n, n-2x)) = 4x+2$, this just means that $x$ is large in terms of $k$ and $g$. Let $t=x/k$. By Theorem~\ref{THM:KGBU}, $\KG(2n,n-2x)$ contains a blow-up of $\KG(2n/t, (n-x)/t)$ with power  $m = {(n-x)(t-1)/t\choose x}$. (Note that $\KG(r,s)$ contains $\KG(r-1,s)$ as a subgraph, and thus, with some neglect of technicalities, we may assume that $x/k$, $2n/t$, $(n-x)/t$, etc. are integers.) Each vertex of $\KG(2n/t, (n-x)/t)$ has degree
${(n+x)/t\choose (n-x)/t} = {(n+x)/t\choose 2x/t}$, which is at most
$$\Delta := ((n+x)/t)^{2x/t} = \Bigl(\frac{nk}x+k\Bigr)^{2k}.$$

In order to apply Theorem \ref{THM:BUT}, we need power $m \ge k(k\Delta)^{2g-4}$ of a graph with chromatic number at least $k$. In the following we assume $x$ is large in terms of $g$ and $k$.
Let us first consider the case when $n > x / \bigl( \frac 12-\frac 1{2k} \bigr)$.
If we write $\tfrac{n}{x} = 2+2z$, this condition implies that $z\ge \tfrac{1}{k}$. If $x>k^2$, then
\begin{equation}
\label{eq:Kneser1}
\frac{n-x}{x} \cdot \frac{t-1}{t} \ge 1+z
\end{equation}
and
\begin{equation}
\label{eq:Kneser2}
 \frac{n}{x} + 1 = 3 + 2z < (1+z)^{x/2}.
\end{equation}
Suppose, moreover, that $x \ge 10gk^2\log k$. Then
\begin{equation}
\label{eq:Kneser3}
 k^{2g(2k+1)} < (1+z)^{x/2}.
\end{equation}
Using inequalities (\ref{eq:Kneser1})--(\ref{eq:Kneser3}), we obtain:
\begin{align*}
   k(k\Delta)^{2g-4} &= k^{2g-3}\Bigl(\frac{nk}x+k\Bigr)^{2k(2g-4)} \\
   &< \Bigl(\frac{(n-x)(t-1)/t}{x}\Bigr)^x \\
   &< {(n-x)(t-1)/t\choose x} = m.
\end{align*}

On the other hand, when $x$ is large enough and $n\le x / \bigl( \frac 12-\frac 1{2k} \bigr)$, then $n-2x \le n/k$ and $\KG(2n,n-2x)$ contains a blow-up of $\KG(k,1) = K_{k}$ with large power.

Thus, we can apply Theorem~\ref{THM:BUT} and conclude that $\KG(2n,n-2x)$ contains a subgraph with girth more than $g$ and chromatic number at least~$k$.
\end{proof}

\end{document}